\documentclass[12pt]{article}
\usepackage{amsmath,amssymb,amsthm,amsfonts,amsbsy, enumerate,mathabx}
\usepackage[dvips]{epsfig}
\usepackage{color}
\usepackage{amsfonts,graphics,epsfig,cite}
\usepackage[margin=1.1in]{geometry}
\usepackage{longtable}
\usepackage{tikz}
\usetikzlibrary{decorations.pathmorphing}
\usetikzlibrary{decorations.pathreplacing}
\usetikzlibrary{patterns}
\usetikzlibrary{arrows}

\usepackage[utf8]{inputenc}

\usepackage[skins,theorems]{tcolorbox}
\tcbset{highlight math style={enhanced,
  colframe=red,colback=white,arc=0pt,boxrule=1pt}}

\theoremstyle{plain}
\newtheorem{theorem}{Theorem}[section]

\newtheorem{lemma}[theorem]{Lemma}
\newtheorem{prop}[theorem]{Proposition}

\newcommand{\Aut}{\operatorname{Aut}}
\newcommand{\wt}{\operatorname{wt}}

\begin{document}

\title{Addressing Johnson graphs, complete multipartite graphs, odd cycles and random graphs}
\author{Noga Alon\footnote{Department of Mathematics, Princeton University, 
Princeton, NJ 08544 and Schools of Mathematics and Computer Science,
Tel Aviv University, Tel Aviv, Israel, {\tt nalon@math.princeton.edu}.}\,,
Sebastian M. Cioab\u{a}\footnote{Department of Mathematical Sciences, University of Delaware, Newark, DE 19707, {\tt cioaba@udel.edu}. }\,, Brandon D. Gilbert\footnote{Department of Mathematical Sciences, University of Delaware, Newark, DE 19707, {\tt brandong@udel.edu}. }\,,\\
 Jack H. Koolen\footnote{School of Mathematical Sciences, University of Science and Technology of China, Wen-Tsun Wu Key Laboratory of the Chinese Academy of Sciences, Hefei, Anhui, China, {\tt koolen@ustc.edu.cn}.} \, and Brendan D. McKay\footnote{Research School of Computer Science, Australian National University, ACT 2601, Australia, {\tt brendan.mckay@anu.edu.au}.}}
\date{\today}
\maketitle

\begin{abstract}
Graham and Pollak showed that the vertices of any graph $G$ can be addressed with $N$-tuples of three symbols, such that the distance between any two vertices may be easily determined from their addresses. An addressing is optimal if its length $N$ is minimum possible.

In this paper, we determine an addressing of length $k(n-k)$ for the 
Johnson graphs $J(n,k)$ and we show that our addressing is optimal 
when $k=1$ or when $k=2, n=4,5,6$, but not when $n=6$ and $k=3$. 
We study the addressing problem as well as a variation of it in 
which the alphabet used has more than three symbols, for other graphs 
such as complete multipartite graphs and odd cycles. We also present 
computations describing the distribution of the minimum length of 
addressings for connected graphs with up to $10$ vertices.
Motivated by these computations we settle a problem of Graham,
showing that most graphs on $n$ vertices have an addressing of length 
at most $n-(2-o(1))\log_2 n$.
\end{abstract}

\section{Introduction}

Let $r\geq 2$ be an integer. A $(0,1,\dots,r-1,*)$-addressing of a graph $G=(V,E)$ is a function $f:V\rightarrow \{0,1,\dots,r-1,*\}^N$ for some natural number $N$ such that for any two vertices $x,y\in V$, the distance between $x$ and $y$ in the graph $G$ equals the number of positions $j$ such that the $j$-th entries of $f(u)$ and $f(v)$ are distinct and neither equals $*$. Let $N_r(G)$ denote the minimum $N$ for which such an addressing is possible. Addressings of length $N_r(G)$ will be called optimal. The distance multigraph $\mathcal{D}(G)$ of the graph $G$ is the multigraph whose vertex set is $V$, where the number of edges between $x,y\in V$ equals the distance in $G$ between $x$ and $y$. It is not too hard to see that $N_r(G)$ equals the minimum number of complete multipartite graphs whose edges partition the edge multiset of the distance multigraph of $G$, where each complete multipartite graph in the partition must have between $2$ and $r$ color classes.

For $r=2$, Graham and Pollak \cite{GP1} conjectured that $N_2(G)\leq n-1$ for any connected graph $G$ with $n$ vertices. This conjecture, also known as {\em the squashed cube conjecture}, was proved by Winkler \cite{W}. Graham and Pollak \cite{GP1} proved the following result (which they attributed to Witsenhausen):
\begin{equation}\label{lowerboundN2}
N_2(G)\geq \max(n_+(D),n_{-}(D)),
\end{equation}
where $D$ is the $|V|\times |V|$ matrix whose entry $(x,y)$ is the distance in $G$ between $x$ and $y$, and $n_+(D)$ and $n_{-}(D)$ denote the number of positive and negative eigenvalues of $D$, respectively. Following Kratzke, Reznick and West \cite{KRW}, an addressing of $G$ of length $\max(n_+(D),n_{-}(D))$ will be called eigensharp. Note that eigensharp addressings are optimal. Graham and Pollak \cite{GP1} proved that complete graphs, trees and odd cycles of order $n$ have eigensharp addressings of length $n-1$ and even cycles have eigensharp addressings of length $n/2$.
Elzinga, Gregory and Vander Meulen \cite{EGVM} proved that the Petersen graph does not have an eigensharp addressing and found an optimal addressing of it of length $6$ (one more than the lower bound \eqref{lowerboundN2}). Cioab\u{a}, Elzinga, Markiewitz, Vander Meulen and Vanderwoerd \cite{CEMVV} gave two proofs showing that the Hamming graphs have eigensharp addressings and started the investigation of optimal addressings for the Johnson graphs. The Johnson graph $J(n,k)$ has as vertices all the $k$-subsets of the set $\{1,\dots,n\}$ and two $k$-subsets $S$ and $T$ are adjacent if and only if $|S\cap T|=k-1$. In this paper, we prove that $N_2(J(n,k))\leq k(n-k)$ by constructing an explicit addressing of $J(n,k)$ with $(0,1,*)$-words of length $k(n-k)$. We answer a question from \cite{CEMVV} and show that $N_2(J(n,2))=2(n-2)$ for $n=5,6$. In the case of $n=6$ and $k=3$, using the computer, we prove that $N_2(J(6,3))=8$ which is smaller than our general bound above. The best known lower bound is $N_2(J(n,k))\geq n$ (see \cite[Theorem 5.3]{CEMVV}).

For $r\geq 3$, Watanabe, Ishii and Sawa \cite{WIS} studied 
$(0,1,\dots,r-1,*)$-addressings and proved that 
$N_r(G)\geq \max(n_{+}(D)/(r-1),n_{-}(D)/(r-1))$. 
Note that the stronger result $N_r(G)\geq \max(n_{+}(D),n_{-}(D)/(r-1))$ 
follows from the work of Gregory and Vander Meulen 
\cite[Theorem 4.1]{GVM} (see also \cite{Sawa}). 
In \cite{WIS}, the first three authors prove that the Petersen graph can be optimally addressed with $(0,1,2,*)$-words of length $4$ and show that $N_r(C_n)=n/2$ for any $n$ even and any $r\geq 3$. For odd cycles, they prove that $N_3(C_{2n+1})=n+1$ for $n\in \{2,3,4\}$ and ask whether this statement is true for larger values of $n$. In this paper, we determine that this is true for $n=5$ and $N_3(C_{11})=6$, but fails for $n\in \{6,7,8,9\}$, where $N_{3}(C_{13})=8, N_{3}(C_{15})=9$, $N_{3}(C_{17})=10$ and $N_3(C_{19})=11$.

For $a,m\geq 1$, let $K(a;m)$ denote the complete $m$-partite graph where each color class has exactly $a$ vertices. The problem of computing $N_2(K(2;m))$ has been investigated by Hoffman \cite{H} and Zaks \cite{Z}.
Using the some small length addressings found by computer for $K(3;3), K(4;4)$ and $K(5;5)$ and a simple combinatorial blow-up argument, we obtain the upper bounds below for any $s\geq 1$:
\begin{align*}
6s\leq N_2(K(3;3s)&\leq 8s-1\\
12s\leq N_2(K(4;4s))&\leq 15s-1\\
20s\leq N_2(K(5;5s))&\leq 24s-1.
\end{align*}
The lower bounds follow from \eqref{lowerboundN2} and unfortunately are quite far from our upper bounds. 

We conclude our paper with an investigation of the typical value of
$N_2(G)$  for connected graphs $G$ on $n$ vertices. We start with
computations describing the distribution of $N_2(G)$ when $G$ 
ranges over all connected graphs with $n\leq 10$ vertices. 
These computations led us to believe that for any
fixed integer $c\geq 1$, almost all connected graphs $G$ of 
order $n$ must have $N_2(G)\leq n-c$, contradicting a suggested conjecture of
Ron Graham from \cite[page 148]{Graham}, where he writes that
it is natural to guess that $N_2(G) =n-1$ for almost all graphs on
$n$ vertices. Motivated by these computations we have been able to prove
our conjecture, showing that in fact $N_2(G) \leq n-(2-o(1))\log_2 n$ for
almost all graphs on $n$ vertices.

\section{Johnson graphs}

For any natural number $m$, we use $[m]$ to denote the set $\{1,\dots,m\}$. Let $n\geq k\geq 1$ be two integers. The Johnson graph $J(n,k)$ has as vertices all the $k$-subsets of the set $[n]$ and two $k$-subsets $S$ and $T$ are adjacent if and only if $|S\cap T|=k-1$. When $k=1$, the Johnson graph $J(n,1)$ is the complete graph $K_n$. When $n=2$, the Johnson graph $J(n,2)$ is the line graph of $K_n$, also known as the triangular graph. Note that the distance between $S$ and $T$ in $J(n,k)$ equals $\frac{|S\Delta T|}{2}=|S\setminus T|=|T\setminus S|$ \cite[p. 255]{BCN}.

To describe our $\{0,1,*\}$-addressing of $J(n,k)$, we need the following function. Let ${[n]\choose k}$ denote the family of all $k$-subsets of $[n]$ and let $\mathcal{P}(X)$ denote the power-set of a set $X$. Define $f:{[n]\choose k}\rightarrow \mathcal{P}(([n]\setminus [k])\times [k])$ as follows. If $S=[k]$, then $f(S)=\emptyset$. If $S\neq [k]$, then let $A=S\setminus [k]=\{x_1,\dots,x_t\}$, with $t\geq 1$ and $n\geq x_1>\dots >x_t\geq k+1$ and let $B=[k]\setminus S=\{y_1,\dots,y_t\}$ with $1\leq y_1<\dots <y_t\leq k$. Define 
\begin{equation}\label{fS}
f(S)=\{(x_1,y_1),\dots,(x_t,y_t)\}.
\end{equation}
For example, if $n=12, k=5$ and $S=\{1,4,6,8,12\}$, then $A=\{12,8,6\}, B=\{2,3,5\}$ and $f(S)=\{(12,2),(8,3), (6,5)\}$.

Our $(0,1,*)$-addressing $a(S,(x,y))$ of each vertex $S$ of $J(n,k)$ with words of length $k(n-k)$ (indexed by the ordered pairs of the form $(x,y)$ with $x\in [n]\setminus [k]$ and $y\in [k]$) is done by the following procedure:
\begin{enumerate}
    \item If $(x,y) \in f(S)$, then $a(S,(x,y))=1$, else
    \item if $\texttt{max}(S) < x$, then $a(S,(x,y))=0$, else
    \item if $(\exists z)\big((z < y) \land ((x,z) \in f(S))\big)$, then $a(S,(x,y))=*$, else
    \item if $y \in S$, then $a(S,(x,y))=0$, else
    \item if $(\exists z)\big((z < x) \land ((z,y) \in f(S))\big)$, $a(S,(x,y))=0$, else
    \item $a(S,(x,y))=*$.
\end{enumerate}

We give below three examples of this addressing in the cases of $J(4,1)$,  $J(5,2)$, and $J(6,3)$. The superscripts in the tables below indicate the rule used for generating that symbol. Since the symbol $1$ can only be generated in step 1, we omit that superscript.

\begin{equation}
\begin{tabular}{|c|c|c|c|c|}
\hline
subset & $(2,1)$ & $(3,1)$ & $(4,1)$ & address \\
\hline
$\{1\}$ &$0^2$ &$0^2$ &$0^2$  & 000\\
\hline
$\{2\}$ &$1$ &$0^2$ &$0^2$  & 100\\
\hline
$\{3\}$ &$*^6$ &$1$ &$0^2$ & *10\\
\hline
$\{4\}$ &$*^6$ &$*^6$ &$1$ & **1 \\
\hline
\end{tabular}
\end{equation}

\begin{equation}
\begin{tabular}{|c|c|c|c|c|c|c|c|}
\hline
subset & $(3,1)$ & $(3,2)$ & $(4,1)$ & $(4,2)$ & $(5,1)$ & $(5,2)$ & address\\
\hline
\{1,2\} &$0^2$ &$0^2$& $0^2$&$0^2$  &$0^2$ & $0^2$ & 000000 \\
\hline
\{1,3\} &$0^4$ &$1$ &$0^2$  &$0^2$  &$0^2$ & $0^2$& 010000 \\
\hline
\{2,3\} &$1$ &$*^3$  &$0^2$  &$0^2$  &$0^2$ & $0^2$& 1*0000  \\
\hline
\{1,4\} &$0^4$ &$*^6$ &$0^4$  &$1$ &$0^2$ & $0^2$ & 0*0100\\
\hline
\{2,4\} &$*^6$ &$0^4$&$1$&$*^3$  &$0^2$ &$0^2$ & *01*00 \\
\hline
\{3,4\} &$*^6$ &$1$ &$1$ &$*^3$  &$0^2$ & $0^2$ & *11*00\\
\hline
\{1,5\} & $0^4$ & $*^6$ & $0^4$ & $*^6$ &$0^4$ & $1$&0*0*01 \\
\hline
\{2,5\} &$*^6$ & $0^4$ & $*^6$ & $0^4$ &$1$ & $*^3$ & *0*01*\\
\hline
\{3,5\} &$*^6$ &$1$  & $*^6$ &$0^5$ &$1$ & $*^3$ & *1*01* \\
\hline
\{4,5\} & $*^6$ &$*^6$ & $*^6$ & $1$ &$1$ & $*^3$ & ***11* \\
\hline
\end{tabular}
\end{equation}
\begin{equation}
\begin{tabular}{|c|c|c|c|c|c|c|c|c|c|c|}
\hline
subset & $(4,1)$ & $(5,1)$ & $(6,1)$ & $(4,2)$ & $(5,2)$ & $(6,2)$ & $(4,3)$ & $(5,3)$ & $(6,3)$ & address\\\hline
\{1,2,3\} &$0^2$ &$0^2$ &$0^2$ &$0^2$ &$0^2$ &$0^2$ &$0^2$ &$0^2$ &$0^2$ &000000000\\\hline
\{1,2,4\} &$0^4$ &$0^2$ &$0^2$ &$0^4$ &$0^2$ &$0^2$ &$1$ &$0^2$ &$0^2$ &000000100\\\hline
\{1,3,4\} &$0^4$ &$0^2$ &$0^2$ &$1$ &$0^2$ &$0^2$ &$*^3$ &$0^2$ &$0^2$ &000100*00\\\hline
\{2,3,4\} &$1$ &$0^2$ &$0^2$ &$*^3$ &$0^2$ &$0^2$ &$*^3$ &$0^2$ &$0^2$ &100*00*00\\\hline
\{1,2,5\} &$0^4$ &$0^4$ &$0^2$ &$0^4$ &$0^4$ &$0^2$ &$*^6$ &$1$ &$0^2$ &000000*10\\\hline
\{1,3,5\} &$0^4$ &$0^4$ &$0^2$ &$*^6$ &$1$ &$0^2$ &$0^4$ &$*^3$ &$0^2$ &000*100*0\\\hline
\{2,3,5\} &$*^6$ &$1$ &$0^2$ &$0^4$ &$*^3$ &$0^2$ &$0^4$ &$*^3$ &$0^2$ &*100*00*0\\\hline
\{1,4,5\} &$0^4$ &$0^4$ &$0^2$ &$*^6$ &$1$ &$0^2$ &$1$ &$*^3$ &$0^2$ &000*101*0\\\hline
\{2,4,5\} &$*^6$ &$1$ &$0^2$ &$0^4$ &$*^3$ &$0^2$ &$1$ &$*^3$ &$0^2$ &*100*01*0\\\hline
\{3,4,5\} &$*^6$ &$1$ &$0^2$ &$1$ &$*^3$ &$0^2$ &$*^3$ &$*^3$ &$0^2$ &*101*0**0\\\hline
\{1,2,6\} &$0^4$ &$0^4$ &$0^4$ &$0^4$ &$0^4$ &$0^4$ &$*^6$ &$*^6$ &$1$ &000000**1\\\hline
\{1,3,6\} &$0^4$ &$0^4$ &$0^4$ &$*^6$ &$*^6$ &$1$ &$0^4$ &$0^4$ &$*^3$ &000**100*\\\hline
\{2,3,6\} &$*^6$ &$*^6$ &$1$ &$0^4$ &$0^4$ &$*^3$ &$0^4$ &$0^4$ &$*^3$ &**100*00*\\\hline
\{1,4,6\} &$0^4$ &$0^4$ &$0^4$ &$*^6$ &$*^6$ &$1$ &$1$ &$0^5$ &$*^3$ &000**110*\\\hline
\{2,4,6\} &$*^6$ &$*^6$ &$1$ &$0^4$ &$0^4$ &$*^3$ &$1$ &$0^5$ &$*^3$ &**100*10*\\\hline
\{3,4,6\} &$*^6$ &$*^6$ &$1$ &$1$ &$0^5$ &$*^3$ &$*^3$ &$0^4$ &$*^3$ &**110**0*\\\hline
\{1,5,6\} &$0^4$ &$0^4$ &$0^4$ &$*^6$ &$*^6$ &$1$ &$*^6$ &$1$ &$*^3$ &000**1*1*\\\hline
\{2,5,6\} &$*^6$ &$*^6$ &$1$ &$0^4$ &$0^4$ &$*^3$ &$*^6$ &$1$ &$*^3$ &**100**1*\\\hline
\{3,5,6\} &$*^6$ &$*^6$ &$1$ &$*^6$ &$1$ &$*^3$ &$0^4$ &$*^3$ &$*^3$ &**1*1*0**\\\hline
\{4,5,6\} &$*^6$ &$*^6$ &$1$ &$*^6$ &$1$ &$*^3$ &$1$ &$*^3$ &$*^3$ &**1*1*1**\\\hline

\end{tabular}
\end{equation}

We give two examples below where the order of our algorithm is significant to the output.

\begin{center}
    $J(4,2)$
\end{center}
\begin{equation}
\begin{tabular}{|c|c|c|c|c|c|c|}
\hline
   subset  & entry & step 1 & step 2 & step 3 & step 4 & step 5 \\ \hline
   $\{2,3\}$ & $(3,2)$  & Fails & Fails & Succeeds & Succeeds & Fails\\
   \hline
\end{tabular}
\end{equation}

\begin{center}
    $J(5,3)$
\end{center}
\begin{equation}
\begin{tabular}{|c|c|c|c|c|c|c|}
\hline
   subset  & entry & step 1 & step 2 & step 3 & step 4 & step 5 \\ \hline
   $\{3,4,5\}$ & $(5,2)$  & Fails & Fails & Succeeds & Fails & Succeeds\\
   \hline
\end{tabular}
\end{equation}

For $S,T\in {[n]\choose k}$, a pair $(x,y)\in ([n]\setminus [k])\times [k]$ is called $(S,T)$-good if 
$$\{a(S,(x,y)),a(T,(x,y))\}=\{0,1\}.$$ Let $c(S,T)$ denote the number of $(S,T)$-good pairs. Our goal is to prove the following result which implies that our procedure on page 2 gives a valid $(0,1,*)$-addressing of $J(n,k)$.
\begin{theorem}
For any $S,T\in {[n]\choose k}, c(S,T)=\frac{|S\Delta T|}{2}=|S\setminus T|=|T\setminus S|$.
\end{theorem}
\begin{proof}
If $S=T$, then the statement is obvious. If $S\neq T$, then the proof follows from Lemma \ref{vert1comp}, Lemma \ref{xmaxymin} and the last sentence of the first paragraph in this section.
\end{proof}

The following results gives a characterization of the $(S,T)$-good pairs and we will use it later in this section.
\begin{lemma}\label{goodpair}
Let $S\neq T\in {[n]\choose k}$ and $(x,y)\in ([n]\setminus [k])\times [k]$. Then 
$$a(S,(x,y))=1 \text{ and } a(T,(x,y))=0$$ 
if and only if the following three conditions are satisfied:
\begin{equation}\label{cond1}
(x,y)\in f(S)\setminus f(T)
\end{equation}
and
\begin{equation}\label{cond2}
\lnot [(\exists z)\big((x<z) \land ((z,y) \in f(T)) \big)]
\end{equation}
and 
\begin{equation}\label{cond3}
\lnot [(\exists z)\big((z<y) \land ((x,z) \in f(T)) \big)]
\end{equation}
\end{lemma}
\begin{proof}
Assume that the conditions \eqref{cond1}, \eqref{cond2} and \eqref{cond3} are true. From \eqref{cond1}, we deduce immediately that $a(S,(x,y))=1$ and $a(T,(x,y))\neq 1$. Thus, $a(T,(x,y))$ is $0$ or $*$.  When evaluating $a(T,(x,y))$, the first step fails since $(x,y)\notin f(T)$. If $\max(T)<x$, then step 2 succeeds, we get $a(T,(x,y))=0$ and we are done. Otherwise, assume that $\max(T)\geq x$. Step $3$ of evaluating $a(T,(x,y))$ fails because \eqref{cond3} is satisfied. If $y\in T$, then step 4 succeeds, $a(T,(x,y))=0$ and we are done. Otherwise, assume that $y\notin T$. There exists $z\in T\setminus [k]$ such that $(z,y)\in f(T)$. By condition \eqref{cond2}, we must have that $z\leq x$. Note that if $z=x$, then we would have that $(x,y)=(z,y)\in f(T)$, contradiction with $(x,y)\in f(S)\setminus f(T)$. Thus, $z<x$. But now step 5 is satisfied and $a(T,(x,y))=0$. Thus, $a(S,(x,y))=1$ and $a(T,(x,y))=0$.

Assume that $a(S,(x,y))=1$ and $a(T,(x,y))=0$. From the definition on the previous page, we deduce that $(x,y)\in f(S)\setminus f(T)$. Thus, \eqref{cond1} is true.

Assume that \eqref{cond2} is not true. Thus, there exists $z_0$ such that $x<z_0$ and $(z_0,y)\in f(T)$. This implies that $y\notin T$. When evaluating $a(T,(x,y))$, step 1 obviously fails. Also, since $\max(T)\geq z_0>x$, step 2 fails as well. Because $a(T,(x,y))=0$, step 3 must also fail. Because $y\in T$, then step 4 must fail. Thus, in order to have $a(T,(x,y))=0$, step 5 must succeed and therefore, there is $z_1<x$ such that $(z_1,y)\in f(T)$. Now $(z_0,y)\in f(T), (z_1,y)\in f(T)$ and $z_0>x>z_1$ provide a contradiction which shows that \eqref{cond2} is true.

Assume that \eqref{cond3} is not true. Thus, there exists $z_0$ such that $z_0<y$ and $(x,z_0)\in f(T)$. Hence, $x\in T$ and $z_0\notin T$. When evaluating $a(T,(x,y))$, step 1 obviously fails. Also, because $x\in T$, we must have that $\max(T)\geq x$ and step 2 fails. The existence of $z_0$ with the above properties implies that step 3 succeeds and $a(T,(x,y))=*$, contradiction with $a(T,(x,y))=0$. Thus, \eqref{cond3} is true and our proof is complete.
\end{proof}

For $S\in {[n]\choose k}$, let $h(S)$ denote the graph with vertex set $[n]$ whose edges are the pairs in $f(S)$. When $S=[k]$, the graph $h(S)$ has no edges and when $S\neq [k]$, $h(S)$ is a matching. For $S\neq T\in {[n]\choose k}$, let $h(S,T)$ denote the multigraph obtained as union of the graphs $h(S)$ and $h(T)$. The non-trivial components of $h(S,T)$ must be cycles or paths. We prove later in this section (Lemma \ref{cycle2}) that the only cycle components possible are cycles of length $2$, but first we will show that the distance in $J(n,k)$ between $S$ and $T$ equals the number of path components in $h(S,T)$.
\begin{lemma}\label{vert1comp}
The set of vertices of degree one in $h(S,T)$ equals $S\Delta T$. Consequently, the number of path components in $h(S,T)$ equals $\frac{|S\Delta T|}{2}=|S\setminus T|=|T\setminus S|$.
\end{lemma}
\begin{proof}
First, we show that $x\in [n]\setminus [k]$ has degree $1$ in $h(S,T)$ if and only if $x\in (S\Delta T)\setminus [k]$.

Assume that $x$ has degree $1$ in $h(S,T)$. Without loss of generality, there exists $y\in [k]$ such $(x,y)\in f(S)\setminus f(T)$. This implies that $x\in S$. Also, we deduce that $x\notin T$, as otherwise there would exist $z$ such that $(x,z)$ is an edge in $h(S,T)$ implying that the degree of $x$ is 2, contradiction. Hence, $x\in S\setminus T\subseteq S\Delta T$. 

Assume that $x\in (S\Delta T)\setminus [k]$. This means that $x\in [n]\setminus [k]$ and without loss of generality, assume that $x\in S$ and $x\notin T$. Because $x\in S$, there exists $y\in [k]$ such that $(x,y)$ is an edge in $h(S)$. The edge $(x,y)$ is the only edge involving $x$ in $h(S)$. Because $x\notin T$, it means that there is no $z$ such that $(x,z)\in f(T)$. Hence, $x$ is not contained in any edges of $h(T)$. Thus, $x$ has degree $1$ in $h(S,T)$.

Secondly, we show that $y\in [k]$ has degree $1$ in $h(S,T)$ if and only if $y\in (S\Delta T)\cap [k]$. 

Assume that $y$ has degree $1$ in $h(S,T)$. Without loss of generality, there exists $x\in [n]\setminus [k]$ such that $(x,y)\in f(S)\setminus f(T)$. This implies that $y\notin S$. Also, $y\in T$, as otherwise there would exist $z\in T$ such that $(z,y)$ is an edge in $h(S,T)$ implying that the degree of $y$ is $2$, contradiction. Hence, $y\in T\setminus S\subseteq S\Delta T$.

Assume that $y\in (S\Delta T)\cap [k]$. Without loss of generality, assume that $y\notin S$ and $y\in T$. Because $y\notin S$, there exists $z\in S$ such that $(z,y)$ is an edge in $h(S)$. This edge is the only edge involving $y$ in $h(S)$. Because $y\in T$, it means that there is no edge involving $y$ in $h(T)$. Hence, $y$ has degree $1$ in $h(S,T)$. This finishes our proof.
\end{proof}

Our goal for the remaining part of this section will be to prove that each path component of $h(S,T)$ contains exactly one good $(S,T)$-pair and that any other component of $h(S,T)$ (isolated vertex or cycle) contains no good $(S,T)$-pairs. 

For the remaining part of this section, let $S\neq T\in {[n]\choose k}$. Let $C$ be a non-trivial component of $h(S,T)$. Define the following:
\begin{align*}
x_{max}(C)&=\max(C\cap ([n]\setminus [k])) \\
y_{max}(C)&=\max(C\cap [k])\\
x_{min}(C)&=\min(C\cap ([n]\setminus [k]))\\ 
y_{min}(C)&=\min(C\cap [k]).
\end{align*}

\begin{lemma}\label{xmaxymin1}
Given any non-trivial component $C$ in $h(S,T)$, at least one of the following statements is true:
\begin{itemize}
    \item The vertex $x_{max}(C)$ has degree one.
    \item The vertex $y_{min}(C)$ has degree one.
    \item The edge $(x_{max}(C),y_{min}(C))$ is contained in both $f(S)$ and $f(T)$.
\end{itemize}
\end{lemma}
\begin{proof}
Assume that each claim above is false. If $x_{max}(C)$ and $y_{min}(C)$ are adjacent, then since $(x_{max}(C),y_{min}(C)\notin f(S)\cap f(T)$, assume that $(x_{max}(C),y_{min}(C))\in f(S)\setminus f(T)$. Because both $x_{max}(C)$ and $y_{min}(C)$ have degree two, there exists $x_0$ and $y_0$ such that $(x_{max}(C),y_0)\in f(T)$ and $(x_0,y_{min}(C))\in f(T)$. Because $x_{max}(C)>x_0$, the definition of $f(T)$ implies that $y_0<y_{min}(C)$, contradiction. If $x_{max}(C)$ and $y_{min}(C)$ are not adjacent (a case that we will see later in Lemma \ref{xmaxyminadj}, never happens), then we can derive a contradiction in a similar manner. 
\end{proof}
\begin{lemma}\label{xminymax1}
Given any non-trivial component $C$ in $h(S,T)$, at least one of the following is true:
\begin{itemize}
    \item The vertex $x_{min}(C)$ has degree one.
    \item The vertex $y_{max}(C)$ has degree one.
    \item The edge $(x_{min}(C),y_{max}(C))$ is contained in both $f(S)$ and $f(T)$.
\end{itemize}
\end{lemma}
\begin{proof}
The proof is similar to Lemma \ref{xmaxymin1} and will be omitted.
\end{proof}
A consequence of Lemma \ref{xmaxymin1} is that the only cycle components of $h(S,T)$ are cycles of length $2$ (double edges joining a pair of vertices).
\begin{lemma}\label{cycle2}
The graph $h(S,T)$ does not contain cycles with more than $2$ vertices.
\end{lemma}
\begin{proof}
If $C$ is a cycle component of $h(S,T)$, then each vertex has a degree two. Thus by Lemma \ref{xmaxymin1}, $x_{max}(C)$ and $y_{min}(C)$ must be doubly adjacent and each only adjacent to one another, and thus must be all the vertices of the cycle.
\end{proof}
This limits the cases of components in $h(S,T)$ to just paths, isolated vertices, and doubly adjacent pairs of vertices. The following lemma uses Lemma \ref{goodpair} to give the first restriction on $(S,T)$-good pairs showing that the only possible good $(S,T)$-pairs are edges involving a vertex of degree one.
\begin{lemma}
No edge $(x,y)$ in $h(S,T)$ with both vertices of degree two is $(S,T)$-good.
\end{lemma}
\begin{proof}
Let $(x,y)$ be an edge with both vertices $x$ and $y$ having degree two. Assume that $(x,y) \in f(S)$. Thus there must exist $y_0$ such that $(x,y_0) \in f(T)$. If $y_0=y$, then $\eqref{cond1}$ is not satisfied. If $y_0 < y$, then $\eqref{cond3}$ is not satisfied. If $y<y_0$, then there must also exist $x_0$ such that $(x_0,y) \in f(T)$. Because $y<y_0$, it must be that $x<x_0$ and $\eqref{cond2}$ is not satisfied. Thus, $(x,y)$ is not $(S,T)$-good.
\end{proof}

\begin{lemma}\label{xmaxyminadj}
For any non-trivial component $C$ of $h(S,T)$, $x_{max}(C)$ and $y_{min}(C)$ are adjacent.
\end{lemma}
\begin{proof}
We prove this result by contradiction. If $x_{max}(C)$ and $y_{min}(C)$ are not adjacent, then assume that $(x_{max}(C),y_0) \in f(S)$ for some $y_0$. It must be that $y_0<y_{min}(C)$, and thus no edge from $f(S)$ could contain $y_{min}$(C). Thus, there is only one edge containing $y_{min}(C)$, say $(x_0,y_{min}) \in f(T)$. As well, by how $f(T)$ is constructed, there are no edges from $f(T)$ that contain $x_{max}(C)$. However, this would result in $x_{min}(C) \leq x_0 < x_{max}(C)$ and $y_{min}(C) < y_0 \leq y_{max}(C)$. Since both $x_{max}(C)$ and $y_{min}(C)$ have degree one, in this path component neither $x_{min}(C)$ nor $y_{max}(C)$ can have degree one and by Lemma \ref{xmaxymin1}, they are doubly adjacent, which can not happen in a path component. This contradiction disproves the assumption and proves the lemma.
\end{proof}
\begin{lemma}\label{xminymaxadj}
For any non-trivial component $C$ of $h(S,T)$, $x_{min}(C)$ and $y_{max}(C)$ are adjacent.
\end{lemma}
\begin{proof}
The proof is similar to the one of the previous lemma and will be omitted.
\end{proof}

\begin{lemma}\label{xmaxymin}
For any path component $C$ in $h(S,T)$, the only edge that is $(S,T)$-good is $(x_{max}(C),y_{min}(C))$.
\end{lemma}
\begin{proof}
By Lemma \ref{xmaxyminadj}, $x_{max}(C)$ and $y_{min}(C)$ are adjacent and without loss of generality, suppose that $(x_{max}(C),y_{min}(C))\in f(S)$. Because $C$ is a path, $(x_{max}(C),y_{min}(C))\notin f(T)$ and \eqref{cond1} is satisfied. Because there is no $x_0$ in $C$ such that $x_{max}(C)< x_0$,\eqref{cond2} is satisfied. Also, there is no $y_0$ in $C$ such that $y_0<y_{min}(C)$ and thus \eqref{cond3} is satisfied. Hence, $(x_{max}(C),y_{min}(C))$ is $(S,T)$-good.

If the component $C$ is a single edge, then we are done. If $C$ has two or more edges, then the only other edge with a degree one vertex is $(x_{min}(C),y_{max}(C))$ as shown by Lemma \ref{xminymax1} and Lemma \ref{xminymaxadj}. Because $C$ is not a single edge, one of $x_{min}(C)$ or $y_{max}(C)$ has degree one and the other has degree two. If $x_{min}(C)$ has a degree of two, there exists $y_0$ such that $(x_{min}(C),y_0)$ is an edge and $(x_{min}(C),y_{max}(C))$ does not satisfy $\eqref{cond2}$ as $y_0<y_{max}(C)$. Otherwise, if $y_{max}(C)$ has a degree of two, there is $x_0$ such that $(x_0,y_{max}(C))$ is an edge. In this case, $(x_{min}(C),y_{max}(C))$ does not satisfy $\eqref{cond3}$, as $x_{min}(C)<x_0$. Hence, $(x_{min}(C),y_{max}(C))$ is not $(S,T)$-good if $C$ has two or more edges.
\end{proof}

\subsection{An improved addressing}

Given that $N_2(J(n,k))=k(n-k)$ for $k=1,n\geq 1$ and for 
$k=2,n\in \{3,4,5,6\}$, it might be tempting to conjecture that 
$N_2(J(n,k))=k(n-k)$ for any integers $n\geq 2k\geq 4$. 
However, this fails for $n=6$ and $k=3$ where we found 
that $N_2(J(6,3))=8$. Under the obvious symmetries, 
there are exactly 246 equivalence classes of addressings 
of length~8, one of which we show below. 
We leave determining $N_2(J(n,k))$ 
for other values of $n$ and $k$ as an open problem.

\begin{equation*}
\begin{tabular}{|c|c|c|}
\hline
subset & address \\\hline
 $\{1, 2, 3\}$ & 0000**** \\
 $\{1, 2, 4\}$ & 0001**** \\
 $\{1, 3, 4\}$ & 01**0000 \\
 $\{2, 3, 4\}$ & 010*010* \\
 $\{1, 2, 5\}$ & 010*10*1 \\
 $\{1, 3, 5\}$ & 01*010*0 \\
 $\{2, 3, 5\}$ & 010011** \\
 $\{1, 4, 5\}$ & 01*110*0 \\
 $\{2, 4, 5\}$ & 010111** \\
 $\{3, 4, 5\}$ & 011**10* \\
 $\{1, 2, 6\}$ & *10*0011 \\
 $\{1, 3, 6\}$ & *1*00010 \\
 $\{2, 3, 6\}$ & *100011* \\
 $\{1, 4, 6\}$ & *1*10010 \\
 $\{2, 4, 6\}$ & *101011* \\
 $\{3, 4, 6\}$ & 11**0*00 \\
 $\{1, 5, 6\}$ & *11**011 \\
 $\{2, 5, 6\}$ & 110*1**1 \\
 $\{3, 5, 6\}$ & *110*11* \\
 $\{4, 5, 6\}$ & *111*11* \\ \hline
\end{tabular}
\end{equation*}

\section{Odd cycles}

Watanabe, Ishii and Sawa \cite{WIS} studied the optimal $(0,1,2,*)$-addressings of various graphs. They observed the following pattern for odd cycles $N_3(C_5)=3, N_3(C_7)=4, N_3(C_9)=5$ and asked the natural question whether $N_3(C_{2n+1})=n+1$ for $n\geq 5$ ?

By computation, we have confirmed these results as well as showing that
$N_3(C_{11})=6$. However, the pattern does not continue further and we have computed $N_3(C_{13})=8$, $N_3(C_{15})=9$, $N_3(C_{17})=10$ and $N_3(C_{19})=11$.
The first four of these values were verified by two independent programs.
Examples of minimal addressings are below. It would be nice to determine $N_3(C_{2n+1})$ in general.
\begin{equation*}
\begin{tabular}{|c|c|c|c|c|c|c|c|c|}
\hline
& $C_5$ & $C_7$ & $C_9$ & $C_{11}$ & $C_{13}$ & $C_{15}$ & $C_{17}$ & $C_{19}$ \\
\hline
1 &000 & 0000 & 00000  & 000000 & 00000000 & 000000000 & 0000000000  & 00000000000 \\
2 & 001 & 0001 & 00001  & 00002* & 00000001& 000000001 & 0000000001 &  00000000001 \\
3 &011 & 0101 & 01001 & 000011 & 00000101 & 000002*01 & 000002*001 &    0200000*001\\
4 & 11* & 0111 & 012*1 & 010011 & 00100101& 000001101 & 0000011001  &    01000001001\\
5 & 2*0 & 111* & 01111 & 012*11 & 0012*101& 001001101& 0010011001  &      010000*1101\\
6 & & *210 & 1111* & 011111& 00111101& 0012*1101 & 0012*11001  &               110*00*1101\\
7 & & 20*0 & *2110& 11111* &00111111 & 001111101 & 0011111001  &               210100*1101\\
8 & &  & 201*0& 11110* &0111111* & 001111111 & 00111112*1  &                        21*100*1111\\
9 & &  & 200*0&*21100 &1*111*10& 01111111* & 0011111111 &                           211100*1121\\
10 & &  &  & 201*00 &**211010 & 01111111* & 011111111* &                               21110111*21\\
11 & &  &  & 200*00 &2*01*010 & 1*1110*10 & 1*1111*110 &                               2111*111*22\\
12 & &  &  & & 2*00*010 & **2110010 &1*1110*110 &                                           2111111**20\\
13 & &  &  & & 020000*0 & 2*01*0010 & **21100110 &                                        2011111**20\\
14 & &  &  & & & 2*00*0010 & 2*01*00110  &                                                      20112**0220\\
15 & &  &  & & & 0200000*0 & 2*00*00110 &                                                      201*2100020\\
16 & & & & & & & 02000001*0 &                                                                         001*2100020\\
17 & & & & & & & 02000000*0 &                                                                         00*022*0020\\
18 & & & & & & & &                                                                                             00*022*0000\\
19 & & & & & & & &                                                                                             000020*0000\\
\hline
\end{tabular}
\end{equation*}

\section{Complete multipartite graphs}

The problem of finding optimal addressings for the complete multipartite graphs is non trivial. Graham and Pollak \cite{GP1} proved that $N_2(T)=|V(T)|-1$ for any tree $T$. This implies that $N_2(K_{1,n})=n$ for any $n\geq 1$. The optimal lengths of $\{0,1,*\}$-addressings of all other complete bipartite graphs were obtained by several authors. 
\begin{theorem}[Fujii-Sawa \cite{FS}, Graham-Pollak \cite{GP1}]\label{N2Kmn}
If $m,n\geq 2$, then 
$$
N_2(K_{m,n})=
\begin{cases}
m+n-1 \text{ if } (m,n)=(2,3), (2,4), (2,6), (3,3), (3,4), (3,5), (3,6), (4,4), (4,5) \\
m+n-2 \text{ otherwise}
\end{cases}
$$
\end{theorem}
We now determine $N_2(K_{a,b,c})$ for several values of $a, b, c$. 
\begin{prop}
For any integer $a\geq 1, N_2(K_{a,1,1})=a+1$.
\end{prop}
\begin{proof}
It is not too hard to see that the eigenvalues of the distance matrix of $K_{a,1,1}$ are $-2$ with multiplicity $a-1$, $-1$ with multiplicity $1$ and $\frac{2a+1\pm \sqrt{(2a+1)^2+8}}{2}$, each with multiplicity $1$. Therefore, the number of negative eigenvalues of this matrix is $a+1$. Inequality \eqref{lowerboundN2} and Winkler's result \cite{W} imply that $N_2(K_{a,1,1})=a+1$.
\end{proof}
For other values of $a,b,c$, we will use the following simple lemmas and Theorem \ref{N2Kmn}
\begin{lemma}\label{23part}
If $a,b,c\geq 1$ are integers, then $N_2(K_{a,b,c})\geq N_2(K_{a+b,c})-1$.
\end{lemma}
\begin{proof}
Adding one column containing exactly $a$ $0$s and $b$ $1$s (corresponding to the partite sets of sizes $a$ and $b$ respectively in $K_{a,b,c}$) to an optimal addressing of $K_{a,b,c}$ will yield an addressing of $K_{a+b,c}$.
\end{proof}

\begin{lemma}\label{plus3}
For any integers $a,b,c\geq 1$, $N_2(K_{a+3,b,c})\leq N_2(K_{a,b,c})+3$.
\end{lemma}
\begin{proof}
Take an optimal addressing $f$ for $K_{a,b,c}$ and make three copies (call them $x, y$ and $z$) of a given vertex $v$ in the $A$ color class. Give the vertices in the new graph $K_{a+3,b,c}$ the following addresses:
\begin{equation*}
g(u)=\begin{cases}
f(v)000 & \text{ if } u=v\\
f(v)011 & \text{ if } u=x\\
f(v)101 & \text{ if } u=y\\
f(v)110 & \text{ if } u=z\\
f(u)*** & \text{ otherwise.}
\end{cases}
\end{equation*}
It can be checked easily that the function $g$ is a valid addressing of $K_{a+3,b,c}$. This proves our assertion.
\end{proof}

Using these lemmas we now prove the following result.
\begin{prop}
For any integers $a,b\geq 2$, $N_2(K_{a,b,1})=a+b-1$.
\end{prop}
\begin{proof}
Combining Lemma \ref{23part} with Graham and Pollak's result involving addressings of stars, we deduce that 
\begin{equation}
N_2(K_{a,b,1})\geq N_2(K_{a+b,1})-1=a+b-1
\end{equation}
for any $a,b\geq 1$.

To prove the upper bound, we use strong induction on $a+b$. By computer, we have found the following optimal addressings of several complete $3$-partite graphs. This takes care of our base case for the induction.
\begin{equation}
\begin{tabular}{|c|c|}
\hline 
$K_{2,2,1}$ & $N_2=3$\\
\hline
A1 & 000\\
A2 & 110\\
B1 & 100\\
B2 & 010\\
C1 & **1\\
\hline
\end{tabular}
\quad
\begin{tabular}{|c|c|}
\hline
$K_{3,2,1}$ & $N_2=4$\\
\hline
A1 & 0000\\
A2 & 0011\\
A3 & 11**\\
B1 & 0*01\\
B2 & 0*10\\
C1 & 10**\\
\hline
\end{tabular}
\quad
\begin{tabular}{|c|c|}
\hline
$K_{4,2,1}$ & $N_2=5$\\
\hline
A1 & 00000\\
A2 & 00011\\
A3 & 011**\\
A4 & 110**\\
B1 & *0*01\\
B2 & *0*10\\
C1 & 010**\\
\hline
\end{tabular}
\end{equation}
\begin{equation}
\begin{tabular}{|c|c|}
\hline
$K_{3,3,1}$ & $N_2=5$\\
\hline
A1 & 0000*\\
A2 & 0011*\\
A3 & 11**0\\
B1 & 0*010\\
B2 & 0*100\\
B3 & 1***1\\
C1 & 10**0\\
\hline
\end{tabular}
\quad
\begin{tabular}{|c|c|}
\hline
$K_{4,3,1}$ & $N_2=6$\\
\hline
A1 & 00000*\\
A2 & 00011*\\
A3 & 011**0\\
A4 & 101**0\\
B1 & **0010\\
B2 & **0100\\
B3 & **1**1\\
C1 & 001**0\\
\hline
\end{tabular}
\quad
\begin{tabular}{|c|c|}
\hline
$K_{4,4,1}$ & $N_2=7$\\
\hline
A1 & 000001*\\
A2 & 000010*\\
A3 & 01**000\\
A4 & 1***001\\
B1 & 0001***\\
B2 & 0010***\\
B3 & 0100**1\\
B4 & 1*00**0\\
C1 & 000000*\\
\hline
\end{tabular}
\end{equation}

Let $a,b\geq 2$ such that $a\geq 5$ and $b\geq 2$. By induction hypothesis, $N_2(K_{a-3,b,1})=(a-3)+b-1$. Lemma \ref{plus3} gives us that $N_2(K_{a,b,c})\leq N_2(K_{a-3,b,c})+3=a+b-1$ which finishes our proof.
\end{proof}

By computer, we have found the following addressings of several other complete $3$-partite graphs. Theorem \ref{N2Kmn} and Lemma \ref{23part} imply that each addressing below is optimal.
\begin{equation}
\begin{tabular}{|c|c|}
\hline
$K_{3,2,2}$ & $N_2=5$\\
\hline
A1 & 00000\\
A2 & 00011\\
A3 & 11***\\
B1 & 010**\\
B2 & *01**\\
C1 & *0001\\
C2 & *0010 \\
\hline
\end{tabular}
\quad 
\begin{tabular}{|c|c|}
\hline
$K_{3,3,2}$ & $N_2=6$\\
\hline
A1 & 000000\\
A2 & 000011\\
A3 & 11****\\
B1 & 0100**\\
B2 & *001**\\
B3 &*010**\\
C1 & *00001\\
C2 & *00010 \\
\hline
\end{tabular}
\quad 
\begin{tabular}{|c|c|}
\hline
$K_{4,2,2}$ &$N_2=6$\\
\hline
A1 & 000000\\
A2 & 000011\\
A3 & 011***\\
A4 & 101***\\
B1 & 0010**\\
B2 & **01**\\
C1 & **0001\\
C2 & **0010 \\
\hline
\end{tabular}
\end{equation}

\begin{equation}
\begin{tabular}{|c|c|}
\hline
$K_{3,3,3}$ & $N_2=7$\\
\hline
A1 & *000000\\      
A2 & *110000\\       
A3 & ***1100\\      
B1 & 0**1000\\    
B2 & 1****10\\   
B3 & 1****01\\      
C1 & 1**1000\\        
C2 & 0100***\\ 
C3 & 0010***\\
\hline
\end{tabular}
\quad 
\begin{tabular}{|c|c|}
\hline
$K_{4,3,2}$ & $N_2=7$\\
\hline
A1 & 0000000\\      
A2 & 0000011\\       
A3 & 011****\\      
A4 & 101****\\    
B1 & 00100**\\   
B2 & **001**\\      
B3 & **010**\\        
C1 & **00001\\ 
C2 & **00010\\
\hline
\end{tabular}
\quad
\begin{tabular}{|c|c|}
\hline
$K_{5,2,2}$ & $N_2=7$\\
\hline
A1 & 0000000\\      
A2 & 0000011\\       
A3 & 0000101\\      
A4 & 0000110\\    
A5 & 0011***\\   
B1 & 0001***\\      
B2 & 0010***\\        
C1 & 01*****\\ 
C2 & 10*****\\
\hline
\end{tabular}
\end{equation}

For $a,m\geq 1$, let $K(a;m)$ denote the complete $m$-partite graph where each color class has exactly $a$ vertices. Thus, $K(1;m)$ is the complete graph on $m$ vertices and $K(a;2)$ is the complete bipartite graph $K_{a,a}$. Determining $N_2(K(2;m))$ is still an open problem and the best results are due to Hoffman \cite{H} (lower bound below) and Zaks \cite{Z} (upper bound):
\begin{equation}
m+\lfloor \sqrt{2m}\rfloor-1\leq N_2(K(2;m))\leq \begin{cases}
(3m-2)/2, \text{ if $m$ is even}\\
(3m-1)/2, \text{ if $m$ is odd}.
\end{cases}
\end{equation}

The following lemma will be used in this section to give upper bounds for $N(K(a;m))$.
\begin{lemma}\label{multipartite}
Let $a,m,s\geq 1$ be integers. If $N_2(K(a;m))\leq t$, then 
\begin{equation}
N_2(K(a;ms))\leq st+s-1.
\end{equation}
\end{lemma}
\begin{proof}
Partition the vertex set of $K(a;ms)$ into $s$ copies of $K(a;m)$. Address these $s$ graphs first using words of length $st$. Then we need to address the remaining edges. This is in essence blow-up version of the complete graph $K_s$ and we need $s-1$ coordinates for this part of the addressing. Thus, $N_2(K(a;ms))\leq st+s-1$.
\end{proof}
If we take $a=m=2$, then it is easy to see that $N_2(K(2;2))=2$. Applying the previous lemma, we get that $N_2(K(2;2s))\leq 2s+s-1=3s-1$ which is the upper bound of Zaks above for $m$ even.

The tables in the Appendix show that $N_2(K(4;4))\leq 14$ and $N_2(K(5;5))\leq 23$. Applying Lemma \ref{multipartite}, we obtain the following upper bounds for $N_2(K(a;as))$ when $a\in \{3,4,5\}$. The lower bounds below are obtained by applying the eigenvalue bound \eqref{lowerboundN2}. The gaps between these bounds are quite large and it would be nice to close them.
\begin{prop}
Let $s\geq 1$ be an integer. Then
\begin{align*}
6s \leq N_2(K(3;3s))&\leq 8s-1\\
12s\leq N_2(K(4;4s))&\leq 15s-1\\
20s\leq N_2(K(5;5s))&\leq 24s-1.
\end{align*}
\end{prop}

\section{Random Graphs: computations and asymptotics}

In \cite{Graham}, Graham uses $r(G)$ for $N_2(G)$ and writes that 
\begin{center}
{\em It is not known how $r(G)$ behaves for random graphs, but it is natural to guess that $r(G)=|G|-1$
for almost all large graphs $G$.}
\end{center}

For $3\leq n\leq 9$, we have computed the distribution of $N_2(G)$ for all connected graphs $G$ on $n$ vertices. Let $\mathcal{F}_n$ denote the family of connected graphs on $n$ vertices. Our results are summarized below. Because every partition the distance multigraph of a connected graph $G$ is a biclique covering of $K_n$, note that $N_2(G)\geq \lceil \log_2 n\rceil$ (see \cite{HHM}).
\begin{equation}\label{compute9n}
\begin{tabular}{|c|c|c|c|c|c|c|}
\hline
$n$ & $|\mathcal{F}_n|$ & $n-1$ & $n-2$ & $n-3$ & $n-4$ & $n-5$\\
\hline
2 & 1 & 1 & 0 & 0 & 0 & 0\\
\hline
3 & 2 & 2 & 0 & 0 & 0 & 0\\
\hline
4 & 6 & 5 & 1 & 0 & 0 & 0\\
\hline
5 & 21 & 17 & 4 & 0 & 0 & 0 \\
\hline
6 & 112 & 67 & 42 & 3 & 0 & 0 \\
\hline
7 & 853 & 316 & 498 & 38 & 1 & 0  \\
\hline
8 & 11117 & 1852 & 7765 & 1469 & 30 & 1 \\
\hline
9& 261080 & 12940 & 159229 & 87094 & 1811& 6 \\
\hline
\end{tabular}
\end{equation}

The computational difficulty of determining $N_2(G)$ increases rapidly as
the order of $G$ or the number of coordinates in addresses becomes
greater.  Our method relies on two symmetry groups, one the symmetries
of the address space and one the automorphisms of~$G$.

The set $\{0,1,*\}^\ell$ is acted on by a group $A_\ell$ of order $2^\ell \,\ell!$,
generated by the $\ell!$ permutations of the coordinates and the $\ell$
elements of order~2 that complement one coordinate.  It is easily checked
that $A_\ell$ preserves distances.  Consequently, we can restrict our
search to addressings that are lexicographically minimal under $A_\ell$.
Fully implementing this restriction would carry too much overhead, so we
limited the pruning to the first three vertices.  For example, we can assume
that the first vertex has an address consisting of some number of $0$s
followed by some number of~$*$s.

After the first three addresses were selected with full pruning by $A_\ell$,
we made lists for each other
vertex $v$ of all the addresses which are the correct distance from each of
the first three addresses.  These were then used in a backtrack search
which processes the vertices in increasing order of their number of
available addresses.  Addresses were stored in one machine word in
a format that allows distances to be calculated in a few machine instructions.
The counts in Table~\ref{compute9n} required about 16 hours of cpu time in total.

Much larger graphs $G$ can only be processed in reasonable time if
their automorphism group $\Aut(G)$ is large.  For any address $\alpha$,
let $\wt(\alpha)$ be the number of $0$s and $1$s in~$\alpha$.  Note that
$\wt(\alpha)$ is preserved by~$A_\ell$, which implies that, if an addressing
of length $\ell$ exists, there is some addressing $f^*$ of length~$\ell$ which is
simultaneously lexicographically minimal under $A_\ell$ and such that
$(\wt(f^*(v_1)),\ldots,\wt(f^*(v_n)))$ is lexicographically minimal under $\Aut(G)$.
We partially implemented the latter restriction as follows: the first vertex $v_1$ has
the smallest value of $\wt(f^*)$ in its orbit under $\Aut(G)$, the second vertex
$v_2$ has the smallest value of $\wt(f^*)$ in its orbit under the stabilizer
$\Aut(G)_{v_1}$, and the third vertex has the smallest value of $\wt(f^*)$ in
its orbit under the two-vertex stabilizer $\Aut(G)_{v_1,v_2}$.
It is likely that this strategy can be improved significantly.

The large number of connected graphs of order 10 (11716571) and the longer time per graph would make it a major operation to do all of those. We ran a random sample of 1/1000 of the connected graphs of order 10 (i.e., 11717 graphs) and obtained this distribution:
\begin{equation}
\begin{tabular}{|c|c|c|c|c|c|}
\hline
$N_2$ & 9 & 8 & 7 & 6 & 5 \\
\hline
\# graphs & 86 & 4105 & 7160 & 363 & 3\\
\hline
\end{tabular}
\end{equation}
These results led us to believe that for
any fixed integer $c\geq 1$, almost all connected graphs $G$ of order 
$n$ have $N_2(G)\leq n-c$. Indeed, we have been able to prove the following
stronger result which confirms this belief and refutes Graham's guess.
We conclude the paper with the statement and its proof.
\begin{theorem}
\label{t51}
For almost all graphs $G$ on $n$ vertices,
$N_2(G) \leq n-(2 -o(1)) \log_2 n$,
where the $o(1)$ term tends to zero as $n$ tends to infinity.
\end{theorem}
\begin{proof}
Let $G=G(n,0.5)$ be the Erd\H{o}s-R\'enyi binomial random graph on
a set $V=\{1,2,, \ldots ,n\}$ of
$n$ labelled vertices. We have to prove that with high probability
(whp, for short), that is, with probability that tends to $1$ as $n$
tends to infinity, $N_2(G)$ is at most $n-(2-o(1)) \log_2 n$.
Let $k=k(n)$ be the largest $k$ so that
$$
{n \choose k} 2^{-{k \choose 2}} \geq 4k^4.
$$
It is easy to check that $k=(2-o(1)) \log_2 n$, and it is not
too difficult to prove that whp $G(n,0.5)$ contains every graph
on $k$ vertices as an induced subgraph. This is proved, for example,
in \cite{Al2}, Theorem 3.1. (We note that we need a much weaker result, 
as we only need to contain one specific graph on $k$ vertices, as
will be clear from the argument below. This can be proved by a second
moment calculation, without using the large deviation techniques applied
in \cite{Al2}. This, however, only effects the $o(1)$-term in our 
estimate, and it is therefore shorter to refer to a proven written 
result without
having to include the second moment computation in the alternative possible
proof.) 

By Theorem 1.1 in  \cite{Al1} there is a biclique covering of the
complete graph $K_k$ on a set $U$ of $k$ vertices 
by at most $\lceil 2 \sqrt k \rceil$ bicliques, so that
each edge is covered once or twice. Fix such a covering, and let
$H$ be the graph on $U$ in which two vertices $u,v \in U$ are adjacent
if the pair $\{u,v\}$ is covered once in the covering above, and are
not adjacent if this pair is covered twice. Since our random graph
$G$ contains, whp, an induced copy of all graphs on $k$ vertices,
it contains an induced copy of $H$. Let $W \subset V$ be the set of
vertices of such a copy. In addition, whp, the diameter of
$G$ is $2$, in fact, every two vertices have at least 
$(1/4-o(1))n$ common neighbors. 
Therefore, whp, the distances in $G$ between
any pair of vertices in $W$ are realized 
precisely by the (at most) $\lceil 2 \sqrt k \rceil$ bicliques we have chosen.
To these bicliques we add now one complete bipartite graph
with vertex classes $W$ and $V-W$. In addition, for each 
vertex $z$ in $V-W$ add a star centered in $z$ whose leaves are all 
vertices of $W$ that are not adjacent in $G$ to $z$, all vertices
in $V-W$ that are not  adjacent to $z$, and all vertices in $V-W$
which are smaller than $z$ and are adjacent to it in $G$. It is easy 
to check that these bicliques realize all distances in $G$, i.e., they partition the distance multigraph of $G$.
The number of these bicliques is at most $n-k+\lceil 2 \sqrt k \rceil +1
=n-(2-o(1)) \log _2 n$.  This completes the proof, and the paper.
\end{proof}

\section*{Acknowledgments}

Noga Alon was supported  
by ISF grant No. 281/17, GIF
grant No. G-1347-304.6/2016 and the Simons Foundation.
Sebastian M. Cioab\u{a} was supported by NSF grants DMS-160078 and CIF-1815922. The research of Brandon D. Gilbert was supported by the University of Delaware Undergraduate Summer Scholar Program. Jack H. Koolen was partially supported by the National Natural Science Foundation of China (Nos.\ 11471009 and 11671376) and by 'Anhui Initiative in Quantum Information Technologies' (Grant No. AHY150200).

\section*{Appendix}
The tables below imply that 
$N_2(K(4;4))\leq 14$ and $N_2(K(5;5))\leq 23$.

\begin{equation}
\begin{tabular}{|c|c|}
\hline
& $K(4;4)$\\
\hline 
A1& *******0000000\\
A2& *******0000011\\
A3& *******0000101\\
A4& *******0000110\\
B1& *******0001***\\
B2& *0000001**0***\\
B3& *1100001**0***\\
B4& 0***1101**0***\\
C1& *******001****\\
C2& 0**01001*0****\\
C3& 0**00101*0****\\
C4& 1**1***1*0****\\
D1& *******01*****\\
D2& 0*****110*****\\
D3& 1100**010*****\\
D4& 1010**010*****\\
\hline
\end{tabular}
\quad
\begin{tabular}{|c|c|}
\hline
& $K(5;5)$\\
\hline 
A1& ******* ******* 000000000\\
A2& ******* ******* 000000011\\
A3& ******* ******* 000000101\\
A4& ******* ******* 000000110\\
A5& ******* ******* 000011***\\
B1& ******* ******* 000001***\\
B2& ******* 0000000 000*10***\\
B3& ******* 0000011 000*10***\\
B4& ******* 0000101 000*10***\\
B5& ******* 0000110 000*10***\\
C1& ******* ******* 0010*****\\
C2& ******* 0001*** **01*****\\
C3& *000000 1**0*** **01*****\\
C4& *110000 1**0*** **01*****\\
C5& 0***110 1**0*** **01*****\\
D1& ******* ******* 01*0*****\\
D2& ******* 001**** *0*1*****\\
D3& 0**0100 1*0**** *0*1*****\\
D4& 0**0010 1*0**** *0*1*****\\
D5& 1**1*** 1*0**** *0*1*****\\
E1& ******* ******* 1**0*****\\
E2& ******* 01***** 0**1*****\\
E3& 0*****1 10***** 0**1*****\\
E4& 1100**0 10***** 0**1*****\\
E5& 1010**0 10***** 0**1*****\\
\hline
\end{tabular}
\end{equation}

\end{document}